\documentclass[11pt, a4paper]{amsart}
\usepackage{graphicx,amssymb}
\usepackage{amsrefs}
\usepackage{hyperref}
\newtheorem{thm}{Theorem}[section]

\newtheorem{prop}[thm]{Proposition}
\theoremstyle{definition}
\newtheorem{defn}[thm]{Definition}
\theoremstyle{remark}
\newtheorem{rem}[thm]{Remark}
\numberwithin{equation}{section}

\newcommand{\CC}{\mathbb C}

\newcommand{\RR}{\mathbb R}
\newcommand{\NN}{\mathbb N}

\renewcommand{\Im}{\operatorname{Im}}
\renewcommand{\Re}{\operatorname{Re}}

\newcommand{\e}{\operatorname{e}}

\begin{document}

\title{On the generalization of Forelli's theorem}%

\author{Jae-Cheon Joo, Kang-Tae Kim and Gerd Schmalz}%
\address{(Joo and Kim) Department of Mathematics, POSTECH, 
Pohang 790-784 The Republic of Korea}%
\email{(Joo) joo@postech.ac.kr}
\email{(Kim) kimkt@postech.ac.kr}
\address{(Schmalz) University of New England, Armidale, 
Australia}
\email{gerd@mcs.une.edu.au}%
\thanks{Research of the first and second named authors is 
supported in part by grant 2011-007831 of the National Research 
Foundation of Korea. The second named author is also supported 
in part by the grant 2011-0030044 (The SRC-GAIA) of the NRF of 
Korea. The third named author is supported by ARC 
Discovery grant DP130103485}%

\begin{abstract}
The purpose of this paper is to present a solution to perhaps
the final remaining case in the line of study concerning 
the generalization of Forelli's theorem on the complex analyticity 
of the functions that are: (\romannumeral 1) $\mathcal{C}^\infty$ smooth
at a point, and (\romannumeral 2) holomorphic along the complex integral 
curves generated by a contracting holomorphic vector field with an isolated
zero at the same point.
\end{abstract}
\maketitle


\section{Introduction}

Among the theorems concerning the complex-analyticity of 
functions of several complex variables, the most exploited 
should be the Hartogs analyticity theorem.  The second may be 
the following theorem of Forelli:

\begin{thm}[Forelli (\cite{For} 1977), Stoll (\cite{Stoll} 1980)]
\label{original}
Let $F \colon B^n \to \CC$ be a complex-valued function defined on 
the unit ball $B^n$ in $\CC^n$. If $F$ satisfies the following two 
conditions:
\begin{itemize}
\item[(\romannumeral 1)] $F \in \mathcal C^\infty(0)$, i.e., for 
every positive integer $k$ there exists an open neighborhood 
$U_k$ of the origin $0$ such that $F \in \mathcal C^k (U_k)$;
\item[(\romannumeral 2)] For every $v \in \CC^n$ with $\|v\|=1$, 
the function $\varphi_v (\zeta) := F(\zeta v)$ defined on the unit 
disc $B^1$ in $\CC$ is holomorphic in the complex variable $\zeta$,
\end{itemize}
then $F$ is holomorphic.
\end{thm}

It was a surprise that the condition (\romannumeral 1) turned out 
impossible to be relaxed to a finite differentiability; consider, 
for instance, the function 
$$
F(z_1,z_2) = 
\begin{cases}
\displaystyle{\frac{\bar z_2}{\bar z_1} \cdot  z_1^{k+2}} 
   & \text{if } z_1 \not= 0 \\
0  & \text{if } z_1=0.
\end{cases}
$$
This function is $\mathcal C^k$ everywhere, satisfies the 
condition (\romannumeral 2) of the hypothesis, but is nowhere 
holomorphic.

On the other hand, the generalizations have occurred recently 
in the following two natural directions. The first direction 
concerns the case that the domain is the union of 
holomorphic discs passing through a single point of the domain. 
In this direction, E. M. Chirka presented the 
complex two dimensional case in \cite{Chirka} and asked whether all 
dimensional generalization is possible.  Responding to the question,
the authors, in the earlier paper \cite{JKS}, presented the following 
result:

\begin{thm}[Chirka 2006, Joo-Kim-Schmalz 2013]
If $\Omega$ is a domain in $\CC^n$ with a ${\mathcal C}^1$ radial
foliation by (non necessarily linear)  holomorphic discs
at a point $p \in \Omega$, then any function 
$F\colon\Omega \to \CC$ satisfying
\begin{itemize}
\item[(\romannumeral1)] 
$F \in {\mathcal C}^\infty(p)$
\item[(\romannumeral2)]
$F$ is holomorphic along the leaves,
\end{itemize}
is holomorphic on $\Omega$.
\end{thm}

This seems to have settled the first direction. Therefore, it 
is natural to shift the focus onto the other direction of 
generalization.  It starts with the re-interpretation of 
the condition (\romannumeral 2) of the original 
Forelli's theorem that the analyticity of the given function 
$F$ along the radial complex lines is equivalent to 
the condition that \it $F$ is holomorphic along 
the complex integral curves of the complex Euler vector field 
$
E=\sum_{j=1}^n z_j \frac{\partial}{\partial z_j} 
$. 
\rm
As soon as this new viewpoint is taken, the following natural
question arises: 
\medskip

\begin{quote}
\bf Question. \it Let $X$ be a holomorphic vector field vanishing 
only at the origin.  Replace the condition {\rm (\romannumeral 2)}  
in the statement of Theorem \ref{original} by the condition:
\begin{itemize}
\item[] \rm
``$f$ is holomorphic along the complex integral \newline 
curves of $X$.''
\end{itemize}
Then, for which $X$ would the conclusion continue to hold?  \rm
\end{quote}
\medskip

\noindent
The answer to this question, in the generic subcase where $X$ is diagonalizable, was given in \cite{KPS}. 

\begin{thm}[Kim-Poletsky-Schmalz \cite{KPS}] \label{thm_KPS}
Let $F\colon B^n\to \mathbb C$ be a function 
defined on the unit open ball in $\CC^n$, and let 
$X = \sum_{k=1}^n \alpha_k z_k \frac{\partial}{\partial z_k}$,
where $\alpha_1, \ldots, \alpha_n$ are complex numbers satisfying
$\alpha_j/ \alpha_\ell>0$ for any $j,\ell \in \{1,\ldots, n\}$.
If $F$ satisfies the following two conditions:
\begin{enumerate}
\item $F\in\mathcal C^\infty(0)$
\item $F$ is holomorphic along the complex integral curves of $X$,
\end{enumerate}
then $F$ is holomorphic on $B^n$.
\end{thm}

We remark that this is modified to fit to the context of this article; 
it was proved originally in \cite{KPS} under the condtion 
that $F$ has a formal Taylor series at the origin, weaker than Forelli's 
original condition that $F \in \mathcal C^\infty(0)$.
\smallskip

It is well-known however that the diagonalizable holomorphic vector fields 
do not always satisfy the additional condition on its coefficients 
specified in the above stated theorem.  But then, it is shown 
in \cite{KPS} with explicit examples that the conclusion fails 
if any of the ratios $\alpha_j/\alpha_k$ should 
take complex non-real, or real-but-negative, values. (See also 
the discussion following Definition 2.1 in Section 2.2.)
\smallskip

On the other hand, in the light of the original 
theorem of Forelli and subsequent generalizations, the case 
of contracting holomorphic vector fields that are not 
diagonalizable should be investigated, since their complex 
integral curves also form a singular foliation at the origin.  
In the case of complex dimension two, all such vector fields, 
up to a change of local coordinates, take the form
$$
X=\alpha \Big( z \frac{\partial}{\partial z}+ (mw+\beta z^m) 
\frac{\partial}{\partial w} \Big)
$$
where $m$ is a positive integer, 
$\alpha \in \CC \setminus \{0\}$ and $\beta \in \CC$.
\medskip

Indeed, the purpose of this paper is to give the answer to this 
seemingly final remaining case.  For the sake of clarity of the
exposition, we present here the version of the main theorem of this 
paper in complex dimension two; the most general all-dimensional 
statement shall be presented in the next section as it needs
further terminology concerning vector fields.

\begin{thm}\label{t;1.2}
Let $X=\alpha z \frac{\partial}{\partial z 
}+ \alpha (mw+\beta z^m) \frac{\partial}{\partial w}$, 
where $m$ is a positive integer, 
$\alpha \in \CC \setminus \{0\}$ and $\beta \in \CC$. 
If  a complex-valued function $F\colon B^2\to\CC$  satisfies the 
conditions:
\begin{itemize}
\item[(1)] $F \in \mathcal C^\infty(0)$
\item[(2)] $F$ is holomorphic along every complex integral curve 
of $X$,
\end{itemize}
then $F$ is holomorphic.
\end{thm}

We remark in passing that the nonzero complex number $\alpha$ 
appearing in the statement above does not play any significant role. 


\section{Contracting fields, Aligned fields and Main Theorem}

\subsection{Contracting holomorphic vector fields}
We start with a holomorphic vector field $X$ defined in an open 
neighborhood of the origin in $\CC^n$.  $X$ is said to be 
{\it contracting at the origin}, if the flow-diffeomorphism, say 
$\Phi_t$, of $\Re X$ for some $t<0$ is contracting at $0$, i.e.,
the map satisfies: (1) $\Phi_t (0)=0$, and (2) every eigenvalue
of the matrix $d\Phi_t|_0$ has absolute value less than 1.

The contracting vector fields have been extensively studied. So
we shall only describe small part of the theory which is directly
related to the theme of this paper. In particular, the 
Poincar\'e-Dulac 
theorem implies that, if $X$ is a contracting holomorphic vector 
field then, up to a change of holomorphic local coordinate system 
at the origin, $X$ can be written in the following form:
\begin{equation}
\label{contract}
X = \sum_{j=1}^n \big(\lambda_j z_j + g_j (z)\big) 
\frac{\partial}{\partial z_j}, 
\end{equation}
where: 
\begin{itemize}
\item[(1)] $0 < \Re \lambda_1 \le \Re \lambda_2 \le \ldots \le
\Re \lambda_n$.
\item[(2)] $g_1 \equiv 0$.
\item[(3)] For every $j \in \{2, \ldots, n\}$,  
$g_j(z)$ is a holomorphic polynomial in the variables 
$z_1,\ldots,z_{j-1}$ only, vanishing at the origin. If the
identity $\lambda_j = \sum_{k=1}^{j-1} m_k 
\lambda_k$ holds for some nonnegative integers $m_k$ 
with $\sum_{k=1}^{j-1} m_k \geq 1$ (called the {\it resonance relation
for $\lambda_j$}), then the condition 
$$
g_j(e^{\lambda_1 \zeta} z_1, \ldots, e^{\lambda_{j-1}\zeta} z_{j-1}) 
= e^{\lambda_j \zeta} g_j (z_1,\ldots, z_{j-1})
$$ 
must also hold.  If no resonance relation holds for 
$\lambda_j$, then $g_j=0$.
\end{itemize}

The natural question to ask at this stage is 
whether Forelli's theorem can 
be generalized to the case of all contracting holomorphic 
vector fields.  The answer is negative; this was already known to be 
impossible even for the diagonalizable case (cf.\ \cite{KPS}).  
We shall see this in further generality in the next section.  

\subsection{Aligned holomorphic vector fields}

The following definition will play the role 
optimal toward the generalization of Forelli's theorem.

\begin{defn}[Aligned fields] \rm \label{aligned}
Let $X$ be a holomorphic vector field of $\CC^n$ contracting 
at the origin. Take its Poincar\'e-Dulac normal form 
(cf.\ \cite{Arnold}, \cite{Sternberg}, \cite{Ueda})
$$
X = \sum_{j=1}^n \big(\lambda_j z_j + g_j (z)\big) 
\frac{\partial}{\partial z_j}
$$
as in (\ref{contract}) above. The vector field $X$ is called 
{\it aligned}, if $\lambda_j / \lambda_k > 0$ for every 
$j, k \in \{1, \ldots, n\}$.
\end{defn}

Notice that, in the Poincar\'e-Dulac normal form of an 
aligned vector field, every variable $z_j$ appears. Note 
also that every $\lambda_j$ can be taken to be positive.
\medskip

If $X$ is {\sl not} aligned on the contrary, then there exists a 
$\mathcal{C}^\infty$ function, say $f$, in a neighborhood of the 
origin satisfying $\overline{ X}f \equiv 0$ 
(i.e., $f$ is holomorphic along every complex integral curve of $X$) 
while $f$ is nowhere holomorphic. Thus the generalization of 
Forelli's theorem fails with such an $X$.  The two-dimensional 
examples given in \cite{KPS} verify this.  For the sake of 
clarity of the exposition, we describe the examples briefly: 
\smallskip

Let the holomorphic vector field $X$ under consideration be not 
aligned. Then one can always extract, from its 
Poincar\'e-Dulac normal form, two distinct complex variables 
$z$ and $w$ (among the variables 
$z_1, \ldots, z_n)$ such that $X$ contains a linear combination 
of $\alpha z \frac{\partial}{\partial z} 
+ \beta w \frac{\partial}{\partial w}$ with the value of the ratio $\alpha/\beta$ not real-positive, and that the remaining part 
of $X$ includes neither $\frac{\partial}{\partial z}$ 
nor $\frac{\partial}{\partial w}$.  [For instance, if $X$ were
given as 
\begin{multline*}
X= z_1 \frac{\partial}{\partial z_1} 
+ (2z_2 +iz_1^2) \frac{\partial}{\partial z_2} \\
+  (1+i) z_3 \frac{\partial}{\partial z_3} 
+ (1-i) z_4 \frac{\partial}{\partial z_4}
+ (2 z_5 - z_3 z_4) \frac{\partial}{\partial z_5},
\end{multline*}
then $z=z_1$ and $w=z_3$ and the vector field we consider is 
therefore $z \frac{\partial}{\partial z} 
+ (1+i) w \frac{\partial}{\partial w}$. Of course
then $\alpha = 1, \beta = 1+i$ in this case.]
\smallskip

We are to show, either in the case of $\alpha/\beta < 0$ or 
in the case of $\alpha/\beta \in \CC \setminus \RR$, that a 
non-holomorphic but $\mathcal{C}^\infty$ smooth function 
$f$ depending only on two complex variables $z, w$ 
(and thus holomorphic in any other variables) can exist 
satisfying $\overline{X}f = 0$, identically.  
If $-t=\alpha/\beta < 0$, then the function
\begin{equation}\label{counter}
f(z_1, \ldots, z_n) = 
\begin{cases} 
\exp \Big( - \frac1{|w|^t|z|} \Big) & \hbox{if } zw \not= 0 \\
0 & \hbox{if } zw = 0
\end{cases}
\end{equation}
is such an example.  
\smallskip

If $\alpha/\beta$ is non-real then,  changing the complex 
parameter $\zeta$ for the flow curve of $X$ by $\tau \zeta$ 
for an appropriate $\tau \in \CC$ and changing also the role 
of $z$ and $w$, we may assume without loss of generality that 
$\alpha = \alpha_1 + i \alpha_2$
with $\alpha_1 >0, \alpha_2 > 0$ and $\beta = t\bar \alpha$ 
with $t>0$.  
If we let $\gamma = \frac1{2\alpha_1} - \frac{i}{2\alpha_2}$, 
then there exists a constant $b>1$ such that the function
$$
f(z_1, \ldots, z_n) := \begin{cases} 
\exp \bigg[\Big(\gamma \log |z|
+\frac{\bar\gamma}{t}\log |w|\Big)^b\bigg] & \hbox{if } zw \not= 0 \\
0 & \hbox{if } zw = 0
\end{cases}
$$
becomes such an example.  For further detailed exposition, 
the reader is invited to read Section 7 of \cite{KPS} (pp. 664-665).
\medskip

\subsection{Statement of Main Theorem}
We now present our main theorem in all dimensions.

\begin{thm} \label{main}
If $F\colon B^n \to \CC$ is a function satisfying the  
conditions:
\begin{itemize}
\item[(\romannumeral 1)] $F \in \mathcal{C}^\infty (0)$, and
\item[(\romannumeral 2)] $F$ is holomorphic along every complex 
integral curve of an aligned holomorphic vector field,
\end{itemize}
then $F$ is holomorphic.
\end{thm}

Notice that this theorem includes Theorem~\ref{thm_KPS} (the main 
theorem of \cite{KPS}).
\smallskip

We are now to present the proof; indeed the rest 
of the paper is devoted to the proof of this theorem.


\section{A formal power series analysis} 
\label{secformal}

We investigate, at this beginning stage, the proof of 
Theorem \ref{main} on the level of formal power series 
which establishes the first step toward the complete proof 
(to be presented in Section 5).  

Recall the usual multi-index notation and the ordering 
as follows: $\alpha = (\alpha_1,\ldots,\alpha_n)$, 
$\beta = (\beta_1,\ldots,\beta_n)$ and $z^\alpha = {z_1}^{\alpha_1}
\cdots {z_n}^{\alpha_n}$.  We also use the {\it length}
$|\beta|:= \beta_1 + \ldots+\beta_n$ and the {\it lexicographic ordering} 
$\prec$ defined by: 
$$
\alpha \prec \beta \Leftrightarrow 
\exists k: ~\alpha_s=\beta_s ~\forall s < k \hbox{ and } \alpha_k < \beta_k.
$$
Denote by $\NN$ the set of nonnegative integers. For any 
$(\lambda_1, \ldots, \lambda_n) \in \CC^n$, define 
the set
$$
A(\lambda_1, \ldots, \lambda_n) := \big\{(m_1,\ldots,m_n) \in \NN^n
\colon \sum_{j=1}^n m_j \geq 1,\,\, \sum_{j=1}^n m_j \lambda_j =0 \big\}.
$$
Then we present:

\begin{prop}\label{p;formal}
Let $X = \sum_{j=1}^n (\lambda_j z_j + g_j (z_1,\ldots,z_{j-1})) 
\frac{\partial}{\partial z_j}$ be a holomorphic vector field, where 
each $g_j$ is a holomorphic polynomial with no constant term. 
If $X$ satisfies the condition
\begin{equation}\label{e;formal}
A(\lambda_1, \ldots, \lambda_n) = \varnothing \hbox{ \rm (the empty set)}, 
\end{equation}
then any formal power series (in the multi-index notation)
$S = \sum_{\alpha,\beta} C_\alpha^\beta z^\alpha \overline{z^\beta}$ 
satisfying $\overline X S = 0$ has to be a holomorphic formal power
series in the sense that $C_\alpha^\beta = 0$ 
for every $\beta$ with $|\beta|>0$.
\end{prop}

\begin{proof} 
Consider the term in $\overline{X}S$ of multi-degree 
$(i_1, \ldots, i_n,j_1,\dots, j_n)$  where $j_1=\cdots= j_{n-1}=0$ 
and $j_n\neq 0$. 
 
Let $\varphi_\nu = \lambda_\nu + g_\nu (z)$ in \eqref{contract}, 
and consider the components of $\overline{X}$ which can now be written 
as $\bar{\varphi}_\nu \frac{\partial}{\partial \bar{z}_\nu}$ 
with $\nu<n$. It can only produce terms 
that contain $\bar{z}_\nu$ with $\nu<n$. So does the 
$\bar{g}_{n}$ term in $\bar{\varphi}_n$. Therefore the 
coefficient of the considered term is equal to
$$ \lambda_n j_{n} C_{i_1, \ldots, i_n}^{j_1, \ldots, j_n}.$$
 
Since $A=\varnothing$, this coefficient vanishes if and only if  
$C_{i_1, \ldots, i_n}^{j_1, \ldots, j_n}=0$ whenever 
$j_1=\cdots=j_{n-1}=0$ and $j_n\neq 0$.

We prove the rest by an induction with respect 
to the lexicographical ordering $\prec$ on multi-indices 
$(j_1,\dots,j_{n})$:
\medskip

\underbar{\sf Initial $(0,\dots,0,j_n)$-th step}: 
already established above. 
\smallskip

\underbar{\sf Assuming the steps prior to 
$(J_1,\dots,J_{n})$}, i.e., that  
$C_{i_1, \ldots, i_n}^{j_1, \ldots, j_n}=0$ for all\break
$(j_1,\dots,j_{n}) \prec (J_1,\dots,J_{n})$, we prove 
\underbar{\sf the $(J_1,\dots,J_{n})$-th step}:
Suppose that $J_k\neq 0$ but $J_\nu=0$ for $\nu<k$. 
Consider the terms in $\overline{X}S$ of multi-degree 
$(i_1,\dots,i_n,J_1\dots,J_n)$.
We show here that such terms cannot be generated by 
$\bar{\varphi}_\nu \frac{\partial}{\partial \bar{z}_{\nu}}$ 
with $\nu<k$ nor by 
$\bar{g}_{k}$ in $\bar{\varphi}_{k} 
\frac{\partial}{\partial \bar{z}_{k}}$. 

The $\bar{g}_{\nu}$ terms in  $\bar{\varphi}_{\nu} 
\frac{\partial}{\partial \bar{z}_{\nu}}$ with $\nu\ge k$, 
if different from zero, either produce 
terms that contain $\bar{z}_{\nu}$ with 
$\nu<k$ or increase the lexicographical multi-degree in  
$\bar{z}_{k},\dots, \bar{z}_n$. In either case they cannot 
produce terms of multi-degree $(i_1,\dots,i_n,J_1\dots,J_n)$.

Hence the coefficient of the term in $\overline{X}S$ 
of multi-degree $(i_1, \ldots, i_n,$ $J_1,\dots, J_n)$ has to 
be equal to
$$ 
\sum_{\nu=k}^n \lambda_{\nu} J_\nu 
C_{i_1, \ldots, i_n}^{J_1, \dots, J_n}.
$$
Since $A=\varnothing$, this coefficient vanishes 
if and only if  $C_{i_1, \ldots, i_n}^{J_1, \ldots, J_n}=0$.
This completes the induction and thus the proof of 
the proposition.
\end{proof}

\begin{rem}
The condition \eqref{e;formal} is essential for the proof. If 
$X =z_1 \frac{\partial}{\partial z_1} 
- t z_2 \frac{\partial}{\partial z_2}$ 
for some positive rational number $t= q/p$, then $X$ understood
as a vector field on $\CC^2$, it is obvious that 
$(q, p) \in A(1,t)$ and hence $A(1,t) \neq \varnothing$; this 
violates Condition \eqref{e;formal}. Also the smooth function 
$F:= |z_1|^{2q} |z_2|^{2p}$ is not holomorphic but satisfies 
the equation $\overline X F \equiv 0$. 
\end{rem}

All holomorphic vector fields contracting at the origin, and hence in 
particular any aligned holomorphic vector fields (cf.\ Definition 
\ref{aligned}), satisfy Condition \eqref{e;formal}. Therefore, 
Proposition \ref{p;formal} implies, in particular, the 
following real-analytic version of generalized Forelli theorem:

\begin{thm} \label{rac;prop}
If $f\colon B^n \to \CC$ is a real-analytic function satisfying 
$\overline X f = 0$ at every point for a holomorphic vector field $X$ 
contracting at the origin, then $f$ is holomorphic.
\end{thm}

\section{A Uniqueness theorem} 

Since it is not known {\it a priori} whether the function $F$ in 
the statement of Theorem \ref{main} should be real-analytic, showing only 
the ``complex-analyticity'' of $F$ on the formal power series level as 
in Section 3 is definitely not sufficient for a proof.  
In order to show that the function
$\partial F / \partial \bar z_j$ itself vanishes for all $j=1,\ldots,n$, 
one needs a new identity theorem for the
appropriate class of functions.  The goal of this section is indeed to 
establish such a principle, whose role will become clear in Section 5 
where we complete the proof of Theorem \ref{main}.

We begin with introducing the appropriate regions in $\CC$. 
Let  $P_1,...,P_n$ be polynomials, that are not identically zero, in the 
single complex variable $\zeta$ and let $\lambda_1,...,\lambda_n$ 
positive real numbers. Consider the open plane-region
$$
D(\underline P,\underline\lambda):=\{\zeta\in \CC  \colon 
|P_j(\zeta)| e^{-\lambda_j\Re\zeta} < 1,\,\, j=1,...,n\}.
$$ 
We call a sequence $\{\zeta_k\}$ in $D(\underline P,\underline\lambda)$ 
{\em admissible}, if 
$$
\lim_{k\rightarrow\infty} |P_j(\zeta_k)| e^{-\lambda_j\Re\zeta_k} =0
$$
for every $j=1,...,n$.  It is obvious that the admissible sequences 
consist of two types of points: 
\begin{itemize}
\item[(1)] $\zeta_k$ tending to the zeros of $P_j$, and 
\item[(2)] $\zeta_k$ with $\Re \zeta_k \to +\infty$.
\end{itemize}
Let $L_r := \{ \zeta \in \CC \colon \Im \zeta = 0, \Re \zeta > r\}$.  
For any $D(\underline P,\underline\lambda)$, there exists a positive 
number $r$ such that $L_r \subset D(\underline P,\underline\lambda)$.   
Denote  by $D^\star (\underline P,\underline\lambda)$ 
the connected component of $D(\underline P,\underline\lambda)$ 
containing this ray $L_r$.  Needless to say, 
$D^\star (\underline P,\underline\lambda)$ is an unbounded 
component of $D(\underline P,\underline\lambda)$.
\medskip

The unique continuation principle we establish is
as follows:

\begin{prop}
\label{p;uniqueness}
Let  $f$ be a bounded holomorphic function on $D(\underline P,\underline\lambda)$. If
\begin{equation}\label{e;ad}
\lim_{k\rightarrow \infty} |f(\zeta_k)| \frac{e^{\lambda_j \ell \Re\zeta_k}}{|P_j(\zeta_k)|^\ell} =0,
\end{equation}
for every $j=1,...,n$, for any nonnegative integer $\ell$, and for any admissible sequence $\{\zeta_k\}$ that are different from any zeros of any $P_j$'s, then $f\equiv 0$ on $D^\star (\underline P,\underline\lambda)$.
\end{prop}

\begin{proof} 
Since $P_j$'s are nontrivial (i.e., not identically zero) polynomials, 
there exists $\epsilon>0$ and $A >0$ such that 
\begin{equation}\label{e;cutoff}
\inf_{\Re \zeta > A \atop 1\le j \le n}|P_j(\zeta)|>\epsilon.
\end{equation}

Set
$D(\underline P,\underline\lambda;A) := 
D(\underline P,\underline\lambda) \cap \{\zeta\in\CC  : \Re\zeta >A\}$. 
Re-ordering the coordinate functions $(z_1,...,z_n)$, we may assume without loss
of generality that
\begin{equation}\label{e;order}
\frac{\mathrm{deg}(P_1)}{\lambda_1} \geq 
\frac{\mathrm{deg}(P_j)}{\lambda_j}, \quad \forall j=1,...,n,
\end{equation}
where $\mathrm{deg} P_j $ represents the degree of the polynomial $P_j$. 
Then a sequence $\{\zeta_k\}$ in $D(\underline P,\underline\lambda;A)$ 
is admissible if and only if 
$$
\lim_{k\rightarrow\infty} |P_1(\zeta_k)| e^{-\lambda_1\Re\zeta_k} =0.
$$

Consider, for every integer $\ell > 0$, the function
$$
g_\ell(\zeta) = f(\zeta)\frac{e^{\lambda_1\ell\zeta}}{P_1(\zeta)^\ell}.
$$ 
Then the function $g_\ell$ is holomorphic on $D(\underline P,\underline\lambda;A)$. 
Now we pose and prove:
\medskip

\noindent\bf Claim 1. \it  $g_\ell$ is bounded on 
$D(\underline P,\underline\lambda;A)$ for every $\ell$.  \rm
\smallskip

Suppose the claim is false for some $\ell$. Then there 
exists a sequence $\{\zeta_k\}$ in 
$D(\underline P,\underline\lambda;A)$ 
such that 
\begin{equation}
\label{e;cont}
|g_\ell(\zeta_k)| \rightarrow\infty
\end{equation}
as $k\rightarrow \infty$.  Since $f$ is bounded, it follows that
$$
\lim_{k\rightarrow\infty}
\frac{e^{\lambda_1\Re \zeta_k}}{|P_1(\zeta_k)|} =\infty,
$$
which obviously implies that
$$
\lim_{k\rightarrow\infty} |P_1(\zeta_k)|
e^{-\lambda_1\Re \zeta_k} =0.
$$
Hence $\{\zeta_k\}$ is an admissible sequence. Then 
by \eqref{e;ad}, we must have $|g_\ell(\zeta_k)|\to 0$
as $k\to \infty$, a contradiction to \eqref{e;cont}.  
Therefore Claim 1 is justified.
\bigskip

Next we pose
\medskip

\noindent\bf Claim 2. \it There exists a constant $\delta_0 >0$ such that 
$|P_1(\zeta)|e^{-\lambda_1\Re\zeta} >\delta_0$ for every 
$\zeta \in \partial D(\underline P,\underline\lambda)\cap 
\{\zeta\in\CC : \Re\zeta \geq A\}$. \rm
\medskip

 Assume the contrary that there exists a sequence 
$\{\zeta_k\}\in \partial D(\underline P,\underline\lambda)
\cap \{\zeta\in\CC : \Re\zeta \geq A\}$ such that 
$|P_1(\zeta_k)|e^{-\lambda_1\Re\zeta_k}\rightarrow 0$ 
as $k\rightarrow\infty$. Since $P_j$ has no zeros on 
$ \{\zeta\in\CC : \Re\zeta \geq A\}$, we see that 
$\Re\zeta_k\rightarrow\infty$ as $k\rightarrow \infty$. Therefore, 
by \eqref{e;order}, we have 
$$|P_j(\zeta_k)|e^{-\lambda_j\Re\zeta_k}\rightarrow 0$$ as 
$k\rightarrow\infty$, for every $j=1,...,n$.  On the other hand,
whenever $\zeta_k\in \partial D(\underline P, \underline\lambda)$, 
the definition of the region $D(\underline P, \underline\lambda)$ implies 
that 
$$
|P_j(\zeta_k)|e^{-\lambda_j\Re\zeta_k} =1
$$ 
for some $j$.  This contradiction proves Claim 2. 
\bigskip 

We now finish the proof of Proposition \ref{p;uniqueness}.
It follows by \eqref{e;cutoff} that 
$$
|P_1(\zeta)| e^{-\lambda_1\Re\zeta} \geq \epsilon e^{-\lambda_1 A}
$$
for every  $\zeta \in D(\underline P,\underline\lambda) 
\cap \{\zeta\in \CC : \Re\zeta =A\}$.  Now, let 
$$
M := \sup_{\zeta \in D(P,\lambda)} |f(\zeta)|, \quad
\hbox{ and } \quad
C := \max\{\delta_0^{-1}, \epsilon^{-1} e^{\lambda_1 A}\}.
$$
Then
$\displaystyle{ 
|g_\ell(\zeta)| = |f(\zeta)|\frac{e^{\lambda_1 \ell\Re\zeta}}{|P_1(\zeta)|^\ell}\leq MC^\ell
}
$
for every 
$\zeta \in \partial D(\underline P, \underline\lambda;A)$. 
Since each $g_\ell$ is a bounded holomorphic function on 
$D(\underline P,\underline\lambda;A)$, the maximum modulus
principle implies that 
$$ 
|g_\ell(\zeta)| = |f(\zeta)|
\frac{e^{\lambda_1 \ell\Re\zeta}}{|P_1(\zeta)|^\ell}
\leq MC^\ell
$$
for every $\zeta \in D(\underline P,\underline\lambda;A)$. 
Now let $\Omega$ be a connected open set satisfying:
\begin{itemize} 
\item[(a)]
$\Omega \subset \{\zeta \in D(\underline P,\underline\lambda;A) 
\colon 
|P_1(\zeta)| e^{-\lambda_1\Re\zeta} < \frac{1}{2C}\}$, and
\item[(b)]
$\Omega \supset L_B := \{\zeta \in \CC \colon \hbox{\rm Im } 
\zeta=0, \Re \zeta >  B\}$, for some $B>A$.
\end{itemize}
Then, for every $\zeta\in \Omega$, we obtain that
$$
|f(\zeta)| \leq \frac{M}{2^\ell}
$$ 
for every positive integer $\ell$.  This implies that 
$f(\zeta)= 0$ for every $\zeta \in \Omega$.  Thus the unique 
continuation principle for holomorphic functions in one complex
variable implies that $f$ vanishes identically on the
unbounded component $D^\star(\underline P,\underline\lambda)$ 
as desired.  The proof of Proposition \ref{p;uniqueness} is 
now complete.
\end{proof}


\section{Proof of Theorem \ref{main}}

We now complete the proof of Theorem \ref{main}, the main result
of this paper. The notation in this section are the same as 
those used in its statement presented in Section 2.
 
\subsection{Holomorphic continuation}
In the next subsection we shall establish that $F$ is holomorphic on 
an open neighborhhood, say $V$, of the origin. Thus, assuming that as an 
established fact, we shall argue that $F$ is indeed holomorphic on the whole 
ball $B^n$.

Since the whole ball $B^n$ is contained in the {\it saturation set}, i.e.,
the maximal open set foliated by the flow curves of $X$ from $V$, 
the complex analyticity of $F$ on the whole ball follows by the 
generalized Hartogs lemma (i.e., Osgood's theorem)
proved in Lemma 6.1 in pp.\ 663--664 of \cite{KPS}, based upon the study
of analytic differential equations \cite{IY}, Ch.\ 1 and 2, e.g.  
\medskip

Thus, toward the proof of Theorem \ref{main}, it only remains to 
establish the complex analyticity of $F$ on some $V$, an open 
neighborhood of the origin, which we shall do in the next subsection.

\subsection{Complex-analyticity of $F$ in an open neighborhood 
of the origin}
Take sufficiently small a neighborhood $\widetilde V$ of the origin on 
which the function $F$ is $\mathcal{C}^2$ smooth. Resizing $\widetilde V$, 
we may assume without loss of generality that $\widetilde V$ is the 
polydisc  $\{z=(z_1,...,z_n)\in \CC^n : |z_j| <1, \;\;j=1,...,n\}$.  Our 
present goal is to establish the complex analyticity of $F$ on an open neighborhood of the origin in $\widetilde V$. 
\smallskip
 
We shall now use the normalization of $X$, for instance following 
\cite{Arnold}, p.\ 187.  If $X$ is the {\it aligned} vector field 
given in the hypothesis of Theorem \ref{main} then, without loss 
of generality, we may assume that 
the vector field $X$ now takes the form
$$
X = \sum_{j=1}^n \big(\lambda_j z_j + h_j (z_1, \ldots, z_{j-1})\big) 
\frac\partial{\partial z_j},
$$ 
with all coefficients $\lambda_j$ of the linear terms in the 
expression of $X$ positive. Furthermore we have, in addition, the 
following:

\begin{itemize}
\item[(1)] $0< \lambda_1 \leq \cdots \leq \lambda_n$ and 
\item[(2)] $h_j$ is a holomorphic polynomial that satisfies
$$
h_j (e^{-\lambda_1 \zeta} z_1, \ldots, e^{-\lambda_{j-1} \zeta} z_{j-1}) =
\e^{-\lambda_j \zeta} h_j ( z_1, \ldots, z_{j-1})
$$  
for every $j=1,\ldots, n$, any $\zeta \in \CC$.
\end{itemize}
Thus, the complex flow-curve of $X$ passing through 
$\eta := (\eta_1, \ldots, \eta_n)$ can be represented by
$$
\zeta \to z(\eta;\zeta) = (z_1 (\eta;\zeta), \ldots, z_n (\eta;\zeta))
$$
where
$$
z_j (\eta;\zeta) = e^{-\lambda_j \zeta} \big(\eta_j + g_j (\eta,\zeta)\big)
$$
with $g_j$ a holomorphic polynomial in $\eta$ and $\zeta$ satisfying
\begin{itemize}
\item $g_1 = 0$,
\item $g_j (\eta; 0) = 0$, and
\item $g_j(\cdot;\zeta)$ depends only upon $\eta_1, \ldots, \eta_{j-1}$
(but not upon $\eta_j, \ldots, \eta_n$),
\end{itemize}
for every $j=1,\ldots, n$.
\smallskip

Note that $\eta = z(\eta;0)$ for every $\eta \in \widetilde V$. Since $X$ is 
a contracting vector field, there exists a neighborhood $V$, with 
$V \subset \widetilde V$, of the origin such that, for every $\eta \in V$, 
the real integral curve $z(\eta; t)$ is defined for every real $t \geq 0$ 
in such a way that its image is contained in $\widetilde V$.
\smallskip

Let $G (\eta, \zeta) := F(z(\eta; \zeta))$. Then the hypothesis on $F$ yields
that $G(\eta, \zeta)$ is a holomorphic function in the variable $\zeta$  
on the region $D_\eta:=D(\underline{P_\eta}, \underline\lambda)$, for 
any $\eta\in V$. Here, of course, 
$$
\underline{P_\eta} (\cdot) 
= (\eta_1 + g_1 (\eta; \cdot),\ldots,\eta_n + g_n (\eta; \cdot)),
$$
and $\underline\lambda = (\lambda_1,\ldots,\lambda_n)$. 
\smallskip

Fix  an arbitrary $\eta^o \in V \setminus \{0\}$. Since we are assuming that
$F$ is $\mathcal{C}^2$ smooth on $\widetilde V$, we have 
$$
\frac{\partial}{\partial\bar\zeta} 
\Big(\frac{\partial G}{\partial\bar\eta_j}\Big) 
= \frac{\partial}{\partial\bar\eta_j} 
\Big(\frac{\partial G}{\partial\bar\zeta}\Big) =0
$$ 
for every $j=1,\ldots,n$. 
This implies that $\partial G/ \partial \bar\eta_j$ is also a bounded 
holomorphic function in $\zeta$ defined on $D_{\eta^o}$. 

The chain rule yields that
$$
\frac{\partial G}{\partial \bar\eta_n} = \frac{\partial F}{\partial\bar z_n} 
e^{-\lambda_n \bar\zeta}.
$$
Since the formal power series of $F$ contains no $\bar z$ terms, the 
function $\partial F/\partial\bar z_n$ is a $\mathcal C^\infty (0)$-function 
with the trivial formal power series representation. Therefore, the function 
$\partial G/\partial \bar\eta_n$ satisfies \eqref{e;ad} on the region 
$D_{\eta^o}$. Consequently, Proposition \ref{p;uniqueness} in Section 4
applies here; it follows therefore that ${\partial G}/{\partial\bar\eta_n}$, 
as well as ${\partial F}/{\partial\bar z_n}$, vanishes identically along the flow 
$\{z(\eta^o,\zeta) \colon \zeta \in D^\star_{\eta^o}\}$, where 
$D^\star_{\eta^o}$ is the unbounded component of $D_{\eta^o}$ 
containing $\mathbb{R}^+ := \{t\in \mathbb R \colon t\geq 0\}$. 

 Moreover for $z_{n-1}$, the chain rule implies that
\begin{eqnarray*}
\frac{\partial G}{\partial\bar\eta_{n-1}} 
&=& \frac{\partial F}{\partial \bar z_{n-1}}\, e^{-\lambda_{n-1} \bar \zeta} 
+ \frac{\partial F}{\partial \bar z_n}\cdot 
\overline{\Big(\frac{\partial g_n}{\partial\eta_{n-1}}\Big)}\cdot 
e^{-\lambda_n\bar\zeta}\\
&=&\frac{\partial F}{\partial \bar z_{n-1}}\, e^{-\lambda_{n-1} \bar \zeta}
\end{eqnarray*}
at every point of the flow curve 
$\{z(\eta^o,\zeta) \colon \zeta \in D^\star_{\eta^o}\}$. 
Proposition \ref{p;uniqueness} applies here again to yield that
$\partial F/\partial \bar z_{n-1} = 0$ on the same flow curve.  
Repeating this process, we arrive at that $\overline\partial F \equiv 0$ 
on the flow curve
$\{z(\eta^o,\zeta) \colon \zeta \in D^\star_{\eta^o}\}$.  
Let $\zeta=0$, in particular, to obtain that 
$\overline\partial F (\eta^o) = 0$. 

Since $\eta^o$ is an arbitrarily chosen point of 
$V\setminus \{0\}$, it
follows that $\overline\partial F = 0$ at every point of 
$V \setminus \{0\}$. Hence $F$ is holomorphic on $V$ also, as 
$F \in \mathcal C^2 (V)$.

Altogether, the proof of Theorem \ref{main} is now complete.
\hfill $\Box$

\begin{rem} \rm 
The proof-arguments just given may appear as if they never used the 
assumption that the holomorphic vector field $X$ must be aligned.  
On the contrary, the assumption was used throughout.  
Notice that Proposition \ref{p;uniqueness} in Section 4, 
which has played a crucial role, is valid only for the 
aligned fields.  Furthermore, Theorem \ref{main} would not hold 
if the vector field $X$ were not assumed to be aligned; see the 
discussion with counterexamples  in Section 2.2 presented 
immediately after Definition 2.1.
\end{rem}

\begin{rem} \rm
The arguments presented just now also prove the main theorem of 
Kim-Poletsky-Schmalz  \cite{KPS}, i.e., Theorem \ref{thm_KPS} in
Section 1 of this paper, but with our condition 
$F \in \mathcal C^\infty(0)$. (N.B. Their original theorem
uses only the existence of formal Taylor series at $0$. But
the method we present in this article needs $F$ to be $\mathcal C^2$ 
in an open neighborhood of the origin $0$.)   
On the other hand the regions $D(\underline P,\underline\lambda)$ 
for the case of \cite{KPS} are simpler; they are just half-planes, 
as their $P_j$'s are constants. 
\end{rem}

\end{document}